\documentclass[12pt]{amsart}
\usepackage{amssymb}
\usepackage{paralist}
\usepackage{fancyhdr}
\usepackage{thmtools}
\usepackage{thm-restate}
\usepackage{hyperref}
\usepackage{cleveref}

\newtheorem{prop}[subsection]{Proposition}
\newtheorem{lem}[subsection]{Lemma}
\newtheorem{rem}[subsection]{Remark}

\newtheorem{defin}[subsection]{Definition}

\newcommand{\Nat}{{\bf N}}

\newcommand{\Rat}{{\bf Q}}
\newcommand{\Prob}{{\bf P}}
\newcommand{\E}{{\bf E}}

\newcommand{\EE}{{\mathcal{E}}}
\newcommand{\VV}{{\mathcal{V}}}

\newcommand{\PP}{{\mathcal{P}}}

\newcommand{\cl}{\textnormal{cl}}
\newcommand{\wcl}{\textnormal{wcl}}
\newcommand{\bpm}{\begin{pmatrix}}
\newcommand{\epm}{\end{pmatrix}}

\begin{document}

\title{Evolving Shelah-Spencer Graphs}
\author{Richard Elwes}

\begin{abstract}
We define an \emph{evolving Shelah-Spencer process} as one by which a random graph grows, with at each time $\tau \in \Nat$ a new node incorporated and attached to each previous node with probability $\tau^{-\alpha}$, where $\alpha \in (0,1) \setminus \Rat$ is fixed. We analyse the graphs that result from this process, including the infinite limit, in comparison to Shelah-Spencer sparse random graphs discussed in \cite{JS} and throughout the model-theoretic  literature. The first order axiomatisation for classical Shelah-Spencer graphs comprises a \emph{Generic Extension} axiom scheme and a \emph{No Dense Subgraphs} axiom scheme. We show that in our context \emph{Generic Extension} continues to hold. While \emph{No Dense Subgraphs} fails, a weaker \emph{Few Rigid Subgraphs} property holds.
\end{abstract}

\maketitle

\section{Introduction}

Random graphs or networks are increasingly important objects of study, in both pure and applied mathematical settings. Such models can be classified into two broad categories: \emph{static} and \emph{evolving}, as observed by (among others) Kumar et al. in their influential work \cite{KRRSTU} on stochastic models of the world wide web\footnote{As observed in \cite{KRRSTU}, there is a regrettable clash in terminology with the work of Erd\H{o}s and R\'enyi who in \cite{ER} discuss the \emph{evolution} of edge density in a different but related sense.}. Models of both types take as input a collection of parameters including the number of nodes $n$, and output a random network of size $n$. The difference is in the role played by $n$. In a static model, $n$ is central throughout the process, and the output is produced directly without proceeding via graphs of smaller size. The famous Watts-Strogatz model from \cite{WS} is an example of such a process; in this case the inputs are $n$, the mean-degree $k$, and the rewiring probability $\beta$. 

In an evolving model, in contrast, the parameter $n$ (assuming it is sufficiently large) plays no role at the start of the process, and the network is constructed one node (or in some models more) at a time. The parameter $n$ provides only a stopping point for the algorithm. Preferential attachment processes, including the celebrated Barab\'asi-Albert model \cite{BA}, are inherently evolving processes in which nodes are added one at a time and connected to pre-existing nodes with probabilities depending on those nodes' current degrees.

Models of both types are of considerable interest to scientists modelling a wide variety of real-world phenomena, including the structure of the web (\cite{KRRSTU}), or biological networks of various kinds (\cite{PPP}) along with dynamical processes on such networks (a long list of examples could be provided here, including models of disease epidemics (\cite{PSV}), racial segregation (\cite{EGIS}), opinion-formation (\cite{VR}), memes within social networks (\cite{GWOL}), biological evolution (\cite{EHN})). However, in many situations evolving models have a clear advantage. After all, very few real-world networks are static.

With this dichotomy in mind, consider the classic Erd\H{os}-R\'enyi model $G(n,p)$, in which each pair of nodes is connected with an edge with probability $p$. This model (unusually) can be viewed as either static or evolving, depending on whether the nodes are all set in place and wired up simultaneously, or inserted in turn with each new node being wired to each previous node with probability $p$.

However, the mathematical properties of the resulting graphs typically hinge on the relationship between $n$ and $p$. For example, if $p \gg \frac{1}{n}$ then $G(n,p)$ will contain a triangle, while if $p \ll \frac{1}{n}$ then $G(n,p)$ will contain no triangle, each  with probability $\to 1$ as $n \to \infty$. This is just one of many \emph{threshold functions} discovered by Erd\H{o}s and R\'enyi in \cite{ER} for properties of $G(n,p)$.

It therefore makes sense, indeed is implicit in the preceding paragraph, to consider the properties of random graphs $G(n, p(n))$ where $p(n)$ is a function of $n$ rather than a fixed constant. Following Shelah and Spencer (see \cite{SS}), we shall be particularly interested in functions of the form $p(n)=n^{-\alpha}$ where $\alpha \in (0,1) \setminus \Rat$. Furthermore, as per the preceding paragraph, it is illuminating to allow $n \to \infty$ in this setting. However, $\left(G(n, p(n)) \right)_{n \in \Nat}$ is unavoidably a sequence of static random graphs: there is no natural way to derive $G(n+1, p(n+1))$ from $G(n, p(n))$. The purpose of the current work is to investigate the following related evolving process: 

\begin{defin} \label{defin:main}
\  \begin{itemize}
\item An \emph{evolving graph process} with monotonically weakly decreasing function $p: \Nat \to [0,1]$ begins with a graph $G_p(1)$ comprising a single node $u_1$. At each time $\tau \geq 2$ create a new graph $G_p(\tau)$ by adding one new node $u_\tau$ to $G_p(\tau-1)$.  Attach $u_\tau$ to each previous node $u_j$ for $j < \tau$, independently, with probability $p(\tau)$.

\item An \emph{evolving Shelah-Spencer process} is an evolving graph process with function $p(\tau) = \tau^{- \alpha}$ for some $\alpha>0$. 
\end{itemize}
\end{defin}

\begin{rem}
The main results in this paper apply to evolving Shelah-Spencer processes with $\alpha \in (0,1) \setminus \Rat$, and the results continue to apply if the initial node $u_1$ is replaced with any initial finite graph. We shall write $G(\tau)$ or $G_\alpha(\tau)$ for $G_p(\tau)$ when the meaning is obvious from context.
\end{rem}

One technique that is possible with evolving processes, and not usually with static ones, is to analyse the graph that emerges by running the process to infinity. Returning to the Erd\H{o}s-R\'enyi process $G(n,p)$ with $p$ a fixed constant, the infinite limit $G(\aleph_0,p)$ is, with probability 1, the famous \emph{Rado graph} (also known as the the \emph{Erd\H{o}s-R\'enyi graph} and simply as \emph{the random graph}). The infinite limits of certain preferential attachment processes have also been analysed, for example in \cite{KK} Kleinberg \& Kleinberg and in \cite{E1}, \cite{E2} by the author.

By definition, it is not possible directly to take the infinite limit of a static process. However, in \cite{SS} Shelah and Spencer deploy the machinery of first order logic (in the language of graphs, which is to say a single irreflexive, symmetric binary relation) to analyse the structures $G(n, n^{-\alpha})$ for fixed $\alpha \in (0,1) \setminus \Rat$. In a breakthrough result, they prove that a zero-one law applies, that is to say every first order sentence will hold in $G \left( n, n^{-\alpha} \right)$ with probability tending either to $0$ or $1$ as $n \to \infty$. They then denote by $T_{\alpha}$ the collection of those sentences which hold with probability approaching $1$. A \emph{Shelah-Spencer graph} (with parameter $\alpha$) is then defined to be a (necessarily infinite) graph which satisfies all the sentences in $T_{\alpha}$.

Shelah-Spencer graphs have proved to be important mathematical structures in a number of respects. Network-theoretically, they provide compelling examples of \emph{sparse} random graphs. Naturally-occurring networks rarely grow with the consistent density exhibited by $G(n,p)_{n \in \Nat}$ for fixed $p$. For example, the probability that two randomly selected webpages are connected by a hyperlink clearly $\to 0$ as $n \to \infty$. Likewise, in an infinite Shelah-Spencer graph, a randomly selected pair of nodes will almost certainly not be joined with an edge.

Shelah-Spencer graphs, and assorted closely related structures, have also occupied a central place in model-theoretic discussion in recent decades. The zero-one law described above implies that $T_{\alpha}$ is a \emph{complete} first order theory. This is the first of many deep logical discoveries about these structures. They are also known to be strictly stable, not finitely axiomatizable, and nearly model complete. They are also closely related to the famous Hrushovski constructions discovered in \cite{EH} which have provided counterexamples to a number of deep model-theoretic conjectures. We shall not delve further into these matters but refer to \cite{BS} for further discussion. 

Shelah-Spencer graphs are inherently infinitary in that they are  not the limits of any known natural evolving procedure. In particular, the outcome of running to infinity the evolving model described in Definition \ref{defin:main} is likely to deviate from $T_{\alpha}$. The purpose of the current work is to establish the extent of that deviation. However, our results are not purely infinitary, but can equally be read as applying at all sufficiently large finite stages of the process.

\subsection{Notation and Preliminaries} \label{subsection:notation}

Given finite subgraphs $A$ and $B$ of some ambient graph $G$, we write $AB$ or $A \cup B$ to mean the induced subgraph of $G$ whose vertex set is the union of those of $A$ and $B$ (i.e including any edges joining $A$ to $B$). If $A$ and $B$ are abstract graphs (i.e. not embedded in some common $G$) then $AB$ or $A \cup B$ is simply the disjoint union of $A$ and $B$. It will be clear from context which is intended.

Given finite graphs $A$ and $B$ (of either type) we write $\VV(B/A)$ for the number of vertices in $AB \setminus A$ (i.e. vertices of $A \cup B$ less those in $A$), and $\EE(B/A)$ for the number of edges in $AB \setminus A$ (i.e. the number of edges in $A \cup B$ less those having both endpoints in $A$), and $\delta_\alpha(B/A) := \VV(B/A) - \alpha \EE(B/A)$. We refer to $B/A$ as a \emph{graph extension} or \emph{rooted graph} with $\VV(B/A)$ many vertices, $\EE(B/A)$ many edges, and \emph{predimension} $\delta_\alpha(B/A)$. (We shall omit the subscript $\alpha$ when it is clear from context.)

The purpose of this notational set-up is that it allows us to consider simultaneously the cases $A \subseteq B$ and where the vertex sets of $A$ and $B$ are disjoint.

We will write $B/A \cong K/H$ to mean that they are isomorphic as rooted graphs, that is to say there is a bijection $f:v(AB) \to v(KH)$ such that $f(v(A)) = v(H)$ and whenever $b_1, b_2 \in v(AB)$ are not both in $v(A)$ then $b_1 b_2$ is an edge in $AB$ if and only if $f(b_1)f(b_2)$ is an edge in $KH$. (The point being that we are unconcerned about the edge relationship strictly within $A$ and $H$.)

We will often be interested in the number of occurrences of a small graph extension $B/A$ within a large (or infinite) graph $G$. We will typically write $\overline{B}/\overline{A}$ for a specific isomorphic copy of $B/A$ where $\overline{A}, \overline{B} \subseteq G$.

We will write $\VV(B)$, $\EE(B)$, and $\delta_{\alpha}(B)$ for $\VV(B/ \emptyset)$, $\EE(B/ \emptyset)$, and $\delta_{\alpha}(B/ \emptyset)$. Note that $\delta_{\alpha}$ is additive:  $\delta_{\alpha}(ABC/A) = \delta_{\alpha}(ABC/AB)+\delta_{\alpha}(AB/A)$. We will also write 
$$d_{\alpha}(B/A):=\max \{ \delta_{\alpha}(B/I) : A \subseteq I \subset AB \}.$$

Recall from \cite{JS} that a finite extension $B/A$ is $\alpha$-\emph{sparse} if $\delta_{\alpha}(B/A)>0$ and $\alpha$-\emph{dense} if $\delta_{\alpha}(B/A)<0$. It is $\alpha$-\emph{safe} if $\delta_{\alpha}(I/A) \geq 0$ for all $A \subseteq I \subseteq AB$, and $\alpha$-\emph{rigid} if $\delta_{\alpha}(B/I) \leq 0$ for all for all $A \subseteq I \subseteq AB$ (i.e. if $d(B/A)<0$). Again, we shall suppress that subscript and simply refer to \emph{sparse}, \emph{dense}, \emph{safe}, \emph{rigid} extensions when $\alpha$ is clear from context.

We collect from \cite{JS} some useful elementary results:

\begin{prop} \label{prop:elem}
\begin{enumerate}[(a)]
\item Every non-safe extension $H/R$ contains a rigid subextension $S/R$.
\item If $H/R$ is neither safe nor rigid, there exists $S \subseteq H$ so that $S/R$ is rigid and $H/S$ is safe.
\item If $H/R$ is rigid and $HX \neq RX$ then $HX/RX$ is rigid.
\end{enumerate}
\end{prop}

\begin{proof}
These are nice exercises, or see Properties 4.1.7, 4.1.15, and 4.1.12 of \cite{JS}. \end{proof}

\begin{lem} \label{lem:safe}
The extension $B/A$ is safe if and only if $d(B/A)=\delta(B/A)$.
\end{lem}

\begin{proof}
This is also essentially contained in \cite{JS}, however the notation of $\delta$ and $d$ are not used there, so we shall spell it out. Suppose first $B/A$ is safe. It is always true that $\delta(B/A) \leq d(B/A)$, so we need to establish the reverse inequality. Whenever $A \subseteq I \subseteq AB$, we have $\delta(B/A) = \delta(B/I) + \delta(I/A)$, and by safeness $\delta(I/A) \geq 0$, so $\delta(B/I) \leq \delta(B/A)$ and thus $d(B/A) \leq \delta(B/A)$ as required. 

Conversely, $\delta(I/A) = \delta(B/A) - \delta(B/I)$ and $\delta(B/I) \leq d(B/A) = \delta(B/A)$ by assumption. Thus $\delta(I/A) \geq 0$. \end{proof}

The final definition of this section (also taken from \cite{SS}) describes a graph-extension $B/A$ being embedded in larger graph $G$ in a particularly nice way:

\begin{defin}
Given finite subgraphs $A$ and $B$ of some ambient graph $G$, and $t \in \Nat$, the extension $B/A$ is  \emph{$t$-generic}, if whenever $C \subseteq G$ and $\VV(C/AB)\leq t$ and $C/AB$ is rigid, then $\EE(C/AB)=\EE(C/A)$.
\end{defin}

\section{Statement of Results}

In the current work, we describe both the infinite limit $G(\infty)$ and all sufficiently large finite graphs formed by the evolving Shelah-Spencer process from Definition \ref{defin:main}. Recall from \cite{JS} that the first order axioms of the Shelah-Spencer theory $T_{\alpha}$ come in two schema \emph{No Dense Subgraphs} and \emph{Generic Extension}, which we take in turn. Let $\mathfrak{S} \models T_{\alpha}$ be a Shelah-Spencer graph.\\

\noindent \textbf{No Dense Subgraphs:} For every dense finite graph $H$ there exists no isomorphically embedded copy of $H$ in $\mathfrak{S}$.\\

\noindent Note that by Proposition \ref{prop:elem}(a), the No Dense Subgraphs axiom is equivalent to:\\

\noindent \textbf{No Rigid Subgraphs:} For every rigid finite graph $H$ there exists no isomorphically embedded copy of $H$ in $\mathfrak{S}$.\\

This axiom fails in our setting. However, we obtain the weaker result that with probability 1 there will only be finitely many copies of each finite rigid graph in $G(\infty)$. To put this another way:

\begin{defin}
Given $r \geq 1$, a vertex is $r$\emph{-irregular} if it is contained in a rigid subgraph of size $\leq r$. 
\end{defin}

\begin{restatable}[Few Rigid Subgraphs]{thm}{denseext} \label{thm:denseext}
For each $r \geq 1$, with probability there exists $C_r>0$ so that for all $T \leq \infty$ there are at most $C_r$ many $r$-irregular vertices in $G(T)$.
\end{restatable}

We turn to the second axiom-schema \emph{Generic Extension} for a Shelah-Spencer graph $\mathfrak{S}$, which transfers directly to our $G(\infty)$:

\begin{restatable}[Generic Extension]{thm}{genericext} \label{thm:genericext}
Suppose $H/R$ is safe and $t \geq 1$. Almost surely, for every $\bar R$ where $|\VV\left( \bar{R} \right)|=|\VV\left(R\right)|$, for all large enough $T \leq \infty$ there exists a $t$-generic copy of $\bar H / \bar R$ in $G(T)$.
\end{restatable}

\begin{rem} \label{rem:manydense}
The natural \emph{Few Dense Subgraphs} axiom scheme dramatically fails in our setting. Consider a dense non-rigid graph $X$. By Proposition \ref{prop:elem}(b) below, this may decomposed as a maximal rigid graph $Y$ with a safe extension $X/Y$. If there is at least one copy of $Y$ in $G(\infty)$ (which is permitted by Few Rigid Subgraphs), then by Proposition \ref{prop:safeext} below, there will be infinitely many copies of $X/Y$ (and thus of $X$) within $G(\infty)$.
\end{rem}

\section{Few Rigid Subgraphs}

The following is the main technical ingredient for our results:

\begin{prop} \label{prop:safeext}
Suppose $H/R$ is a finite graph extension with $d=d(H/R)$ and $\delta=\delta(H/R)$. Suppose $\bar R$ is a subgraph of some $G \left( \tau_0 \right)$ with $|\VV \left(\bar R\right)|=|\VV \left(R\right)|$. 
\begin{enumerate}[(a)]
\item If $H/R$ is not rigid the expected number of instances of $H/ \bar R$ contained in $G(T) \setminus G(\tau_0)$ (or equivalently in $G(T))$ is $\Theta \left(  T^{d} \right)$, where $d>0$.

\item If $H/R$ is rigid, the expected number of instances of $H/ \bar R$ contained entirely in $G(\infty) \setminus G(\tau_0)$ is $\Theta \left( \tau_0^{\delta} \right)$ (where $\delta<0$). 

\item If $H/R$ is rigid, the expected total number of instances of $H/ \bar R$ contained in $G(\infty)$ is  positive and finite, with a value depending on $H/R$ and $G \left( \tau_0 \right)$.
\end{enumerate}
\end{prop}

\begin{proof}

We shall compute the expected number of copies of $H/ \overline{R}$ completely contained in $G(T) \setminus G(\tau_0)$. 

We shall deal separately with the case where $R = \emptyset$, that is to say we count occurrences of a graph $H$. First suppose $R \neq \emptyset$. Fix, temporarily, an enumeration $v_1, \ldots, v_n$ of the vertices of $H/R$ and write $e_{j+1}$ for the number of edges connecting $v_{j+1}$ to $\left\{ v_1, \ldots, v_j \right\} \cup R$. 

Given fixed $\overline{v_1  \ldots v_{n-1} R}$ in $G \left( \tau_{n-1} \right)$, the probability that a new node $u_{\tau}$ where $\tau>\tau_{n-1}$ forms a copy of  $v_n / \overline{v_1 \ldots v_{n-1} R}$ is
\begin{equation} \label{equation:edgeprob} \tau^{-\alpha e_n} \cdot \left(1 - \tau^{-\alpha} \right)^{r+n-1 - e_n}. \end{equation}

Setting, for example, $C:=\left(1 - 2^{-\alpha} \right)^{r+n-1 - e_n}$, this probability exceeds $C \cdot \tau^{- \alpha e_n}$ for all $\tau$. Thus expected number of copies of $v_n / \overline{v_1 \ldots v_{n-1} R}$ in $G(T) \setminus G \left( \tau_{n-1} \right)$  is $ \Theta \left( \sum_{\tau = \tau_{n-1}+1}^T \tau^{- \alpha e_n} \right)$. 

Iterating this, and noting that whenever $\beta > 0$ it will hold that $$\int_j^S s^{-\beta} ds \geq \sum_{s=j}^S s^{-\beta} \geq \int_{j+1}^{S+1} s^{-\beta} ds \geq 2^{-\beta} \int_{j}^S s^{-\beta} ds$$
 the expectation we are seeking is asymptotically 
\begin{equation} \label{equation:multint}
\Theta \left( \sum_{v_1, \ldots, v_n} \int_{\tau_0}^T \int_{\tau_1}^T \ldots \int_{\tau_{n-1}}^T \tau_1^{-\alpha e_1} \tau_2^{-\alpha e_2} \ldots \tau_n^{-\alpha e_n} d \tau_n \ldots d\tau_2 d\tau_1 \right)
\end{equation}
where the sum is over all enumerations $v_1, \ldots, v_n$ of $H/R$.

In the case $R= \emptyset$, similar reasoning gives:
\begin{equation} \label{equation:multint3}
\Theta \left( \sum_{v_1, \ldots, v_n} \int_{\tau_0}^T \int_{\tau_1}^T \ldots \int_{\tau_{n-1}}^T 1 \cdot \tau_2^{-\alpha e_2} \ldots \tau_n^{-\alpha e_n} d \tau_n \ldots d\tau_2 d\tau_1 \right)..
\end{equation}

We consign the analysis of these integrals to Appendix A where it is established that both expressions above are
\begin{equation} \label{equation:Theta}
\Theta \left( \sum_{R \subseteq S \subseteq H} C_S \cdot T^{\delta(H/S)} \cdot \tau_0^{\delta(S/R)} \right) \end{equation}
where $C_S$ are constants with in particular $C_{S_{d}}>0$ where $\delta \left( H/S_d \right)$ is maximal, and in the rigid case $C_H>0$.

In the rigid case, since $\delta \left(H/S \right)<0$ for all $S$, this is therefore $\Theta \left(\tau_0^{\delta(H/R)} \right)$ as $T \to \infty$ giving part (b). Otherwise, it is $\Theta \left( T^d \right)$.

We also need to consider copies of $H/\overline{R}$ which are split across $G(\tau_0)$ and $G(T) \setminus G(\tau_0)$. This is bounded above by the number of copies of $H/ S$ in $G(T) \setminus G(\tau_0)$, summed across the constant number of isomorphism types of $S/R$ where $R \subset S \subset H$ and the bounded number of $\bar{S}/ \bar{R}$ in $G(\tau_0)$. In the case where $H/R$ is not rigid, this number will in any case be $O \left(  T^{d} \right)$, giving part (a). Thus we only worry about the rigid case. Here, it is sufficient to observe that this number is non-negative and bounded in terms of $\tau_0$, giving part (c). \end{proof}

\begin{rem} \label{rem:missedpoints}
Assuming that $H/R$ is safe, Proposition \ref{prop:safeext}(a) continues to hold if we count extensions $H/R$ which avoid any finite set of vertices. This amounts to removing boundedly many points from the range of the integral in Expression (\ref{equation:multint}) or (\ref{equation:multint3}).
\end{rem}

We may now recall and prove 

\denseext*

\begin{proof}
This follows from Proposition \ref{prop:safeext}(c) since there are finitely many isomorphism types of rigid graphs of size $\leq r$, and with probability 1 finitely many copies of each in $G(\infty)$. Of course, also, the number of $r$-irregular vertices in any $G(T)$ is at most that in $G(\infty)$.
\end{proof}

\section{Generic Extensions}

Unless otherwise stated, the context for the  all the following is the infinite graph $G(\infty)$. In \cite{JS}, (a) in the following definition is an important concept. However, we shall work with the weaker notion (b):

\begin{defin}
Given $t \geq 1$ and a finite set of vertices $X$
\begin{enumerate}[(a)]
\item the \emph{$t$-closure} of $X$, denoted $\cl_t(X)$, is the minimal set $Y \supseteq X$ such that there exists no rigid extension $Z/Y$ where $\VV(Z/Y) \leq t$.
\item the \emph{weak $t$-closure} of $X$, denoted $\wcl_t(X)$, is the union of all rigid extensions $Z/X$ where $\VV(Z/X) \leq t$ (with all induced edges included).
\end{enumerate}
\end{defin}

Theorem 4.3.2 from \cite{JS} states that in their setting, given $r,t \geq 1$, there exists $K$ so that with probability $1$, for all $\bar{X} = \left( x_1, \ldots, x_r \right)$ we have $\left| \cl_t(\bar X) \right| \leq K$.

This is not apparently attainable in our context. The argument from \cite{JS} identifies $K$ so that if $\cl_t(\bar X)>K$ then $\cl_t(\bar X)$ is necessarily a dense graph, contradicting the No Dense Subgraphs axiom. As discussed in Remark \ref{rem:manydense}), this axiom fails badly in our setting, and we see no way around this obstacle. However, the following much weaker result will be sufficient:

\begin{lem} \label{lem:wcl}
Given $\bar{X}= \left( x_1, \ldots, x_r \right)$, with probability 1, $\wcl_t \left(\bar{X} \right)$ is finite.
\end{lem}

\begin{proof}
There are only finitely many isomorphism types of relevant extensions $Z/X$, and thus the result follows immediately from Proposition \ref{prop:safeext}(c).
\end{proof}

Next we need the following strengthening of Proposition \ref{prop:safeext}(a) for safe extensions:

\begin{prop} \label{prop:conc}
Suppose that $H/R$ is a safe graph extension with $d = d(H/R)$ and $\bar R$ is a tuple of size $|R|$ in $G \left( \tau_0 \right)$. Write $N(T) = N_{H/\bar R}(T)$ for the number of distinct copies of $H/\bar R$ in $G(T)$. The  probability 1, as $T \to \infty$, $N(T)  = \Theta \left( T^{d} \right)$. \end{prop}

\begin{proof}
We have seen in Proposition \ref{prop:safeext} that $\E \left(N(T) \right) = \Theta \left( T^{d} \right)$. Thus we need to show that $N(T)$ is concentrated around its mean, for which we use the machinery developed by Kim and Vu in \cite{KV}. In particular, we shall apply Corollary 4.1.3 of that paper. 

We need to verify that there exist $K, \gamma>0$ so that for all subextensions $I/R$ and all instances $\bar I$ of $I/\bar R$, we have $$\frac{\E \left(N_{H/\bar I}(T) \right)}{\E \left(N_{H/\bar R}(T) \right)} < K \cdot  T^{-\gamma}.$$
Well, $\E \left(N_{H/\bar I}(T) \right) = \Theta \left( T^{d(H/I)} \right) = \Theta \left( T^{\delta(H/J)} \right)$ for some $I \subseteq J \subset H$ and $\E \left(N_{H/\bar R}(T) \right) = \Theta \left( T^{\delta(H/R)} \right)$ by Lemma \ref{lem:safe}, so 
$$\frac{\E \left(N_{H/\bar I}(T) \right)}{\E \left(N_{H/\bar R}(T) \right)}  = \Theta \left( T^{\delta(H/J) - \delta(H/R)}\right) = \Theta \left(T^{-\delta(J/R)} \right)$$
by the additivity of $\delta$. Since $H/R$ is safe, $\delta(J/R) > \gamma >0$ where $\gamma:= \frac{1}{2} \cdot \min_{R \subset J' \subseteq H} \left\{ \delta(J'/R) \right\}$, giving the result.
\end{proof}

\begin{rem}
Proposition \ref{prop:conc} fails if $H/R$ is non-safe. For instance, take $R = \emptyset$ and $H = A \cup B$ where $A$ is rigid and $B/A$ is safe, and $0< -d(A) < d(B/A)$. Then the expected number of occurrences of $H$ is $\Theta \left( T^{d(B/A) + d(A)} \right) \to \infty$ but there is a non-vanishing probability that all $G(T)$ will contain zero copies of $A$ and thus of $H$. This fact somewhat complicates the proof of Theorem \ref{thm:genericext} below.   
\end{rem}

Our final goal is to establish the existence of generic extensions. The following will be useful.

\begin{defin}
Suppose $H/R$ is safe. A minimally rigid extension $K/HR$ is \emph{loose over} $R$ if $KH/R$ is safe. It is \emph{tight over} $R$ otherwise.
\end{defin}

We are now in a position to recall and prove

\genericext*

\begin{proof}
Suppose $H/R$ has associated parameters $v,e, \delta$. By Lemma \ref{lem:wcl} almost certainly, $\wcl_{t+v}\left( \bar R \right)$ is finite. Thus by Lemma \ref{lem:safe}, Remark \ref{rem:missedpoints}), and Proposition \ref{prop:conc} for all large enough finite $T$, almost surely there will be $\Theta \left( T^{\delta } \right)$ many extensions $H /\bar R$ avoiding $\wcl_{t+v} \left( \bar R \right)$. We limit our consideration to these \emph{semi-generic} copies of $H$.

Now fix $\bar{R}$ and suppose $K/ H R$ is a minimally rigid extension of size $\leq t$ and that $K$ witnesses non-$t$-genericity in that $K/H R$ includes at least one edge between $K$ and $H$. Our goal is to show that almost surely in all large enough $T$, there will be $o \left( T^\delta \right)$ many copies of $H/ \bar{R}$ admitting an extension to $K$. Since there are boundedly many isomorphism types of such $K/HR$, this will establish that almost surely for all large enough $T$, there are $o \left( T^\delta \right)$ many non-generic copies of $H/\bar{R}$. 

First we consider the case where $K/HR$ is tight, that is to say $KH/R$ is not safe. We shall show that for all such tight extensions $K/R$ is rigid, meaning $\overline{K} \subseteq \wcl_t \left(\overline{R} \right)$.  By Proposition \ref{prop:elem}(a), $KH/R$ has a rigid subextension $K_1/R$. Then any $\overline{K_1} \subseteq \wcl_{t+v} \left( \bar R \right)$, and thus all semi-generic copies of $H$ are disjoint from all such $\overline{K_1}$. Combining this with the fact that $K_1/R$ is rigid, by Proposition \ref{prop:elem}(c), we see that $K_1/HR$ is rigid. But $K/HR$ is minimally rigid, so it follows that $K_1 = K$, and in particular $K/R$ is rigid as claimed.

Thus with probability 1, there will be finitely many copies of $\overline{K}/\overline{R}$ in $G(\infty)$, so we analyse now those copies of $H$ which grow entirely after all such $\overline{K}$ are completed.

Consider the extension $H / \overline{K R}$. Notice that this is safe, for if not it has a rigid sub-extension and therefore meets $\wcl_{t+v} \left( \bar R \right)$, contradicting our assumption of semi-genericity. Under the assumption that we are considering $H/R$ which fail to be $t$-generic, there is at least one edge connecting $H$ to $K$. Thus $H / K R$ has parameters $v, e', \delta'$ where $e'\geq e+1$ and $\delta' \leq \delta - \alpha$. 

Thus by Proposition \ref{prop:conc}, almost surely there are $O \left( T^{\delta - \alpha} \right)$ many copies of $H / \overline{K R}$ in $G(T)$ and thus $\Theta \left( T^{\delta} \right)$ many copies of $H / \overline{R}$ which are not which are not instances of $H / \overline{K R}$. Combining this with the fact that there are boundedly many isomorphism types $H / \overline{K R}$ and finitely many copies of $\overline{K} / \overline{R}$, and since a completed unjoined copy cannot become joined subsequently, this establishes the almost sure existence of $\Theta \left( T^{\delta} \right)$ many copies of $H$ in all $G(T)$ where $T \leq \infty$, which are both semi-generic and for which there is no tight $K/\overline{HR}$ to witness non-genericity. We restrict our attention to these copies of $H/ \overline{R}$.\\

We move onto the case is where $K/HR$ is loose, that is $KH/R$ is safe. Call $I/HR$ a \emph{partial} extension if $HR \subseteq IHR \subset KHR$ where $K/HR$ is minimally rigid and loose. Notice for all such $I$, both $I/HR$ and $IH/R$ are safe.

By Proposition \ref{prop:conc} again, for any $\overline{R}$, the number of copies of $H / \overline{R}$ in $G(T)$ is almost surely $\Theta \left( T^{\delta(H/R)} \right)$. At the same time, the number of copies of $KH / \overline{R}$ in $G\left( T \right)$ is almost surely $\Theta \left( T^{\delta(KH/R)} \right)= o\left( T^{\delta(H/R)} \right)$. Thus the number of copies of $H / \overline{R}$ in $G\left( T \right)$ which have no complete extension to $K$ in $G\left( T \right)$ is almost surely $\Theta \left( T^{\delta(H/R)} \right)$.

Furthermore, for each isomorphism type of a partial extension $I/HR$, the number of instances of $IH/\overline{R}$ in $G\left( T \right)$ is almost surely $\Theta \left( T^{\delta(IH/R)} \right)$, and thus for each $\overline{H}$, the number of instances of $I/\overline{HR}$ in $G\left( T \right)$ is almost surely $\Theta \left( T^{\delta(I/HR)} \right)$.

For each such partial extension $\overline{IHR}$, the expected number of completions to minimally rigid loose $K$ in $G \left(\infty \right)$ is $\Theta \left( T^{\delta(K/IHR)} \right) \to 0$ as $T \to \infty$. Thus the probability that $\overline{IHR}$ fails to complete to any $K$ exceeds $1 - C T^{\delta(K/IHR)} \to 1$ for some $C>0$.

The probability, given $\overline{HR}$ in $G(T)$, that all instances of $I/\overline{HR}$ in $G(T)$ subsequently fail to complete in $G(\infty)$ exceeds, for some $D>0$: 
$$\left( 1 - C T^{\delta(K/IHR)} \right)^{D T^{\delta(I/HR)}} \sim 1 - O\left( T^{\delta(K/HR)} \right) \to 1.$$

Taking the product of finitely many expressions of this type to take into account different isomorphism types of $K$ and $I$, we find that the probability that any $\overline{HR}$ fails to complete to any minimally rigid loose $K/ \overline{HR}$ tends to $1$, and hence the probability that at least one doesn't is, in the limit, $1$.
\end{proof}

\section{Further Questions}

To our knowledge, this paper represents the first time the graphs $G = G_{\alpha} (\infty)$ have been studied, and Theorems \ref{thm:denseext} and \ref{thm:genericext} represent only a starting point for investigation. So we close by mentioning some directions for further enquiry.

\begin{enumerate}
\item \label{item:irreg} Writing $\cl(A) := \bigcup_{i=1}^\infty cl_t(A)$, it is natural to investigate the probability that $\cl(\emptyset)$ is infinite, that is to ask whether the total number of irregular vertices ($t$-irregular for any $t$) is infinite. We expect that this will hold with probability $1$. This is certainly not trivial, but some finessing of the estimates in Appendix A may be enough to provide a proof of this.\\

\item \label{item:allirreg} A stronger result would be to establish that $\cl(\emptyset)$ is cofinite in $G$, or even that $\cl(\emptyset)=G$. That is to say, it is conceivable that with probability 1 every vertex is $t$-irregular for some $t$, although we expect new tools will be required to answer this question in either direction.\\

\item It important to stress that while \textbf{No Rigid Subgraphs} and \textbf{Generic Extension} provide a complete first order axiomatisation of the theory of Shelah-Spencer graphs, it is certainly not the case that \textbf{Few Rigid Subgraphs} and \textbf{Generic Extension} do so in our context. There surely cannot be any simple axiomatisation, given the variety of finite rigid graphs which may or may not arise. However, one might ask whether incorporating $\textrm{Diag} \left(\cl(\emptyset) \right)$ and stipulating No Rigid Subgraphs outside $\cl(\emptyset)$ would provide a (non-first order) axiomatisation, or whether there are other important properties to be found.\\

\item \label{item:forking} Relatedly, and as indicated in the introduction to this paper, Shelah-Spencer graphs are important in first order model theory, as examples of \emph{stable} graphs. That is to say, there is a natural notion of two finite sets of vertices being \emph{independent} over a third set, which satisfies some natural and powerful axioms. The  Erd\H{o}s-R\'enyi random graph, meanwhile, satisfies the related property of being \emph{supersimple}. It is natural to ask then, whether anything can be said about forking independence within our graphs $G$.\\

\item The theory of Shelah-Spencer graphs had essentially two separate beginnings: the work of Shelah and Spencer (\cite{SS}) as briefly described earlier, and separately the work of Baldwin and Shi \cite{BShi}, following the seminal work of Hrushovski \cite{EH}. The latter authors constructed these graphs as the limits of a variant of Fra\"iss\'e amalgamation on the class of finite sparse graphs. It was in \cite{BS} that Baldin and Shelah established the equivalence of the two approaches. One might therefore wonder whether some variant of amalgamation ``over $\cl\left(\emptyset) \right)$'' (in some sense to be determined) might function similarly in our context.\\

\item Shelah-Spencer graphs fit into the broader theory of threshold functions of graphs $G(n,p(n))$. Here a celebrated theorem Bollob\'as and Thomason in \cite{BT} states that every monotonic increasing graph property $\PP$ has a threshold function. Recall that a property of graphs $\mathcal{P}$ is \emph{monotonic increasing} if whenever $G$ and $G^*$ are graphs where $\VV(G) = \VV(G^*)$ and $\EE(G) \subseteq \EE(G^*)$, if $G \in \PP$ then also $G^* \in \PP$. Given a graph property $\PP$, and a function $p^*: \Nat \to [0,1]$ we say that $p^*$ is an \emph{threshold function} for $\PP$ if for all $p: \Nat \to [0,1]$
$$\lim_{n \to \infty} \Prob \left(G(n, p(n)) \in \PP \right) = \left\{ \begin{array}{ll}
0 & \textrm{if } \frac{p(n)}{p^*(n)} \to 0 \\
1 & \textrm{if } \frac{p(n)}{p^*(n)} \to \infty
\end{array} \right.$$

So long as $p(n)$ is a monotonically decreasing function of $n$, it is immediate by monotonicity of $\PP$ that the second part of the definition of a threshold function transfers into the evolving context. The first does not. For instance, $p^*(n) = n^{-\frac{2}{3}}$ is the threshold function for containing a clique of size 4. However, consider the evolving process with parameter $\frac{3}{4}> \frac{2}{3}$. The first four nodes already determine that the probability that $G_{\frac{3}{4}} \left( \infty \right)$ contains a 4-clique exceeds $\left(2 \cdot 3^2 \cdot 4^{3} \right)^{-\frac{3}{4}}>0$.

Nevertheless, we expect that the theory of threshold functions to continue to apply with the weaker condition ``$<1$'' replacing ``$=0$''. If so, one might investigate under what circumstances the stronger condition of ``$=0$'' applies (one might expect it to hold, for instance, in the case of conntectedness, with a threshold of $\frac{\ln n}{n}$).

\end{enumerate}

\section*{Appendix A} 

Here we analyse the integral (\ref{equation:multint}) from the proof of Proposition \ref{prop:safeext}:
\begin{equation*}
I_n \left( \tau_0, \alpha_1, \ldots, \alpha_n \right) := \int_{\tau_0}^T \int_{\tau_1}^T \ldots \int_{\tau_{n-1}}^T \tau_1^{-\alpha_1} \tau_2^{-\alpha_2} \ldots \tau_n^{-\alpha_n} d \tau_n \ldots d\tau_2 d\tau_1
\end{equation*}
aiming to show that it satisfies condition (\ref{equation:Theta}):
\begin{equation} 
\Theta \left( \sum_{R \subseteq S \subseteq H} C_S \cdot T^{\delta(H/S)} \cdot \tau_0^{\delta(S/R)} \right) \end{equation}
where $C_S$ are constants with in particular $C_{S_{d}}>0$ where $\delta \left( H/S_d \right)$ is maximal, and in the rigid case $C_H>0$. We shall then show that the conclusion applies to integral (\ref{equation:multint3}).

We treat $T$ as fixed and write $\alpha_i := \alpha e_i$. The base case is 
\begin{equation} \label{equation:base}
I_1 \left[ \tau_0, \alpha_1 \right] = \frac{T^{1 - \alpha_1}}{1 - \alpha_1} - \frac{\tau_0^{1 - \alpha_1}}{1 - \alpha_1}. \end{equation}
The relevant recurrence relation is: \begin{equation} \label{equation:intrecur}
I_{n+1}\left[ \tau_0, \alpha_1, \ldots, \alpha_{n+1} \right] = \int_{\tau_0}^T \tau_1^{-\alpha_1} \times I_n \left[\tau_1, \alpha_2, \ldots, \alpha_{n+1} \right] d \tau_1.
\end{equation}

We shall abbreviate $I_{n}\left( \tau_0, \alpha_1, \ldots, \alpha_{n} \right)$ as $I_n$ (and similarly for $C_j^n$ in the following), i.e. we suppress the variables when (and  only when) they are the canonical ones.

\begin{lem} \label{lem:CJs}
For all $n \geq 1$ and $0 \leq j \leq n$ there are constants $C_j^n = C_j^n \left[\alpha_1, \ldots \alpha_n \right] \in {\bf R}$ such that :
\begin{equation} \label{equation:cj}
I_n = \sum_{j=0}^n C_j^n \times T^{(n-j)-\left( \alpha_{j+1} + \ldots + \alpha_n \right)} \times \tau_0^{j - \left( \alpha_1 + \ldots + \alpha_j \right) }. \end{equation}
Furthermore the following recurrence relation holds:
$$C_{j+1}^{n+1} \left[\alpha_1, \ldots \alpha_{n+1} \right] = \frac{- C_j^n \left[\alpha_2, \ldots \alpha_{n+1} \right]}{  (j+1) - \left(\alpha_1 + \ldots + \alpha_{j+1} \right) }$$
\end{lem}

\begin{proof}
Suppose inductively for some $n \geq 1$ and all $0 \leq j \leq n$ that there are suitable constants $C_j^n$. Then, writing $\left(C_j^n\right)'$ for $C_j^n \left[\alpha_2, \ldots \alpha_{n+1} \right]$,  by the recurrence relation (\ref{equation:intrecur}) we get 
\begin{align*}
I_{n+1} & = \int_{\tau_0}^T \tau_1^{\alpha_1} \times \sum_{j=0}^n \left(C_j^n\right)'  \times T^{(n-j)-\left( \alpha_{j+2} + \ldots + \alpha_{n+1} \right)} \times \tau_1^{j - \left( \alpha_2 + \ldots + \alpha_{j+1} \right) } d \tau_1 \\
& = \sum_{j=0}^n \left(C_j^n\right)' \times T^{(n-j)-\left( \alpha_{j+2} + \ldots + \alpha_{n+1} \right)} \times \int_{\tau_0}^T \tau_1^{j - \left(\alpha_1 + \alpha_2 + \ldots + \alpha_{j+1} \right) } d \tau_1 \\
& = \sum_{j=0}^n \frac{\left(C_j^n\right)' \times T^{(n+1)-\left( \alpha_1 + \ldots + \alpha_{n+1} \right)}}{ \left( j+1 \right) - \left(\alpha_1 + \alpha_2 + \ldots + \alpha_{j+1} \right) } \\
& \ \ \ \ \ \ \ \ - \sum_{j=0}^n \frac{\left(C_j^n\right)' \times T^{(n-j)-\left( \alpha_{j+2} + \ldots + \alpha_{n+1} \right) } \times \tau_0^{(j+1) - \left( \alpha_1 + \ldots \alpha_{j+1} \right)}}{ \left( j+1 \right) - \left(\alpha_1 + \alpha_2 + \ldots + \alpha_{j+1} \right) } 
\end{align*}
This establishes inductively that Equation \ref{equation:cj} is indeed the correct form, and the second sum provides the required recurrence relation for $C_j^n$. 
\end{proof}

\begin{rem} \label{rem:czero} 
The first sum in the expression above reflects the fact that $$C_0^{n+1} = - \sum_{j=1}^{n+1} C_j^{n+1}$$
which is what we expect from considering the case $\tau_0 = T$, where clearly $I_n = 0$.
\end{rem}

\begin{prop}
For all $n \geq 1$ and $0 \leq j \leq n$: 
$$C_{j}^{n} = \frac{ \left(- 1 \right)^{j}}{ \left( \prod_{k=1}^{n-j} \left(k - \left( \alpha_{j+1} + \ldots + \alpha_{j+k} \right) \right)  \right)  \left( \prod_{i=1}^j \left( i - \left( \alpha_{j+1-i} +\ldots + \alpha_{j} \right) \right)   \right)}.$$
\end{prop}

\begin{proof}

The base case of $n=1$ and $0 \leq j \leq 1$ is established in Equation (\ref{equation:base}) above. Suppose now that the result holds for some particular $n$ and all $0 \leq j \leq n$. Then 
$$\left(C_{j}^{n} \right)' = \frac{ \left(- 1 \right)^{j}}{ \left( \prod_{k=1}^{n-j} \left(k - \left( \alpha_{j+2} + \ldots + \alpha_{j+1+k} \right) \right)  \right) \left( \prod_{i=1}^j \left( i - \left( \alpha_{j+2-i} +\ldots + \alpha_{j+1} \right) \right)   \right)}.$$
and so 
$$C_{j+1}^{n+1} = \frac{ \left(- 1 \right)^{j+1}}{ \left( \prod_{k=1}^{(n+1)-(j+1)} \left(k - \left( \alpha_{j+2} + \ldots + \alpha_{j+1+k} \right) \right)  \right)  \left( \prod_{i=1}^{j+1} \left( i - \left( \alpha_{j+2-i} +\ldots + \alpha_{j+1} \right) \right)   \right)}$$
as required.\\

Thus we are left with the case of $C^{n}_0$, which we approach from a different angle. Notice that

\begin{align*}
C^{n}_0 \times T^{n - \left( \alpha_1 + \ldots + \alpha_n \right)} &= I_{n+1}\left[0, \alpha_1, \ldots, \alpha_{n+1} \right]\\
& = \int_{0}^T \int_{\tau_1}^T \ldots \int_{\tau_{n-1}}^T \tau_1^{-\alpha_1} \tau_2^{-\alpha_2} \ldots \tau_n^{-\alpha_n} d \tau_n \ldots d\tau_2 d\tau_1\\
& = \int_{0}^T \int_0^{\tau_{n}} \ldots \int^{\tau_2}_0 \tau_1^{-\alpha_1} \tau_2^{-\alpha_2} \ldots \tau_n^{-\alpha_n} d \tau_1 d\tau_2 \ldots d \tau_n
\end{align*}
since in both cases the function $\tau_1^{-\alpha_1} \tau_2^{-\alpha_2} \ldots \tau_n^{-\alpha_n}$ is being integrated over all tuples $(\tau_1, \tau_2, \ldots, \tau_n)$ where 
$0 \leq \tau_1 \leq \tau_2 \leq \ldots \leq \tau_n  \leq T$.

From this formulation, it is an easy exercise to show that
$$C_0^n = \frac{1}{\left(1 - \alpha_1 \right)\cdot \left(2 - \left(\alpha_1 + \alpha_2 \right) \right) \cdot \ldots \cdot \left(n - \left( \alpha_1 +\ldots + \alpha_n \right) \right) }$$
as required. \end{proof}

We note in passing that combining this result with the sum for $C_0^{n+1}$ observed in Remark \ref{rem:czero} above yields an interesting and non-obvious identity.

There is one further step to complete the analysis of (\ref{equation:multint}) and conclude that when we permute the vertices $v_1, \ldots, v_n$ and sum, its output satisfies condition (\ref{equation:Theta}). We need to establish that (in the notation of (\ref{equation:Theta})), the term $C_{S_d}>0$ where $S_d$ is such that $R \subseteq S_d \subset H$ and $\delta \left( H/S_d \right) = d(H/R)$. Happily this is easy to see. The term $C_{S_d}$ arises as a sum (over different permutations of $v_1, \ldots, v_n$) of constants of the form $$C_{j}^{n} = \frac{ \left(- 1 \right)^{j}}{ \left( \prod_{k=1}^{n-j} \delta\left(A_k /S_d \right)  \right)  \left( \prod_{i=1}^j \delta \left( S_d / B_i \right)   \right)}$$
for certain $S_d \subset A_k \subseteq H$ and $R \subseteq B_i \subset S_d$. Since $\delta \left(H / S_d \right) = \delta \left(H / A_k \right) + \delta \left(A_k / S_d \right)$, by hypothesis on $S_d$ it follows that $\delta \left(A_k / S_d \right)>0$. Likewise $\delta \left(H / B_i \right) = \delta \left(H / S_d \right) + \delta \left(S_d / B_i \right)$  meaning that $\delta \left(S_d /B_i \right)<0$ and $C^n_j>0$.

By the same reasoning, in the rigid case $C^n_n>0$, meaning that after summing over permutations $v_1, \ldots, v_n$ we have $C_H>0$ as required.\\

Finally we return to integral \ref{equation:multint3}:
$$J_{n}\left[\tau_0, \alpha_2, \ldots,\alpha_{n}  \right] = \int_{\tau_0}^T \int_{\tau_1}^T \ldots \int_{\tau_{n-1}}^T 1 \cdot \tau_2^{-\alpha_2} \ldots \tau_n^{-\alpha_n} d \tau_n \ldots d\tau_2 d\tau_1.$$
We can analyse this as follows (keeping the notation from above):
\begin{align*}
& J_{n}\left[\tau_0, \alpha_2, \ldots,\alpha_{n+1}  \right] = \int_{\tau_0}^T I_{n-1} \left( \tau_1, \alpha_2, \ldots, \alpha_{n} \right) d \tau_1\\
= & \left(\left(C^{n-1}_0 \right)'+\sum_{j=1}^{n-1} \frac{\left(C^{n-1}_j \right)'}{(j+1) - \left(\alpha_2+ \ldots + \alpha_{j+1} \right)} \right) T^{n - \left(\alpha_2+ \ldots + \alpha_n \right)}\\
& \ \ - \left(C^{n-1}_0 \right)' \cdot T^{(n -1) - \left(\alpha_2+ \ldots + \alpha_n \right)} \cdot \tau_0 \\
& \ \ - \sum_{j=1}^{n-1} \frac{\left(C^{n-1}_j \right)' \cdot T^{(n-1-j) - \left(\alpha_{j+2}+ \ldots + \alpha_{n} \right)} \cdot \tau_0^{(j+1) - \left(\alpha_{2}+ \ldots + \alpha_{j+1} \right)}}{(j+1) - \left(\alpha_2+ \ldots + \alpha_{j+1} \right)}.
\end{align*}

So we may write 
\begin{align*}
J_{n}\left[\tau_0, \alpha_2, \ldots,\alpha_{n}  \right] = & D^n_0 \cdot T^{n - \left(\alpha_{2}+ \ldots + \alpha_{n} \right)} 
+ D^n_1  \cdot  T^{(n-1) - \left(\alpha_{2}+ \ldots + \alpha_{n} \right)} \cdot \tau_0 \\
& \ \ + \sum_{j=2}^n D^n_j \cdot T^{(n-j) - \left(\alpha_{j+1}+ \ldots + \alpha_{n} \right)} \cdot \tau_0^{j - \left(\alpha_{2}+ \ldots + \alpha_{j} \right)}
\end{align*}
where
$$D^n_1 = - \left(C^{n-1}_0 \right)' = \frac{-1}{\prod_{k=1}^{n-1} \left(k - \left( \alpha_{2} + \ldots + \alpha_{k+1} \right) \right)}$$
and for $2 \leq j \leq n$,
\begin{align*}
& D^n_j = \frac{- \left( C^{n-1}_{j-1} \right)'}{j - \left(\alpha_{2}+ \ldots + \alpha_{j} \right)}\\ & = \frac{ \left(- 1 \right)^{j}}{ \left( \prod_{k=1}^{n-j} \left(k - \left( \alpha_{j+1} + \ldots + \alpha_{j+k} \right) \right)  \right) \left( \prod_{i=1}^{j-1} \left( i - \left( \alpha_{j+1-i} +\ldots + \alpha_{j} \right) \right)   \right) \left( {j - \left(\alpha_{2}+ \ldots + \alpha_{j} \right)}\right)}
\end{align*}
and 
\begin{align*}
& D^n_0 = \left(C^{n-1}_0 \right)'+\sum_{j=1}^{n-1} \frac{\left(C^{n-1}_j \right)'}{(j+1) - \left(\alpha_2+ \ldots + \alpha_{j+1} \right)}\\
& \hspace{-3cm} = \frac{1}{\prod_{k=1}^{n-1} \left(k - \left( \alpha_{2} + \ldots + \alpha_{k+1} \right) \right)}\\
& \hspace{-2.8cm} + \sum_{j=1}^{n-1} \frac{ \left(- 1 \right)^{j}}{ \left( \prod_{k=1}^{n-1-j} \left(k - \left( \alpha_{j+2} + \ldots + \alpha_{j+1+k} \right) \right)  \right) \left( \prod_{i=1}^j \left( i - \left( \alpha_{j+2-i} +\ldots + \alpha_{j+1} \right) \right) \right) \left( (j+1) - \left(\alpha_2+ \ldots + \alpha_{j+1} \right)   \right) }.\\ 
\end{align*}

Again, we need to establish that $C_{S_d}>0$ (in the notation of (\ref{equation:Theta}), reinterpreted to our new context) where $S_d$ is such that $S_d \subset H$ and $\delta \left( H/S_d \right) = d(H)$. This time, the term $C_{S_d}$ arises as a sum (over different permutations of $v_1, \ldots, v_n$) of constants of the form $$D_{j}^{n} = \frac{ \left(- 1 \right)^{j}}{ \left( \prod_{k=1}^{n-j} \delta\left(A_k /S_d \right)  \right)  \left( \prod_{i=0}^{j-1} \delta \left( S_d / B_i \right)   \right)}$$
for some $S_d \subset A_k \subseteq H$ and $B_i \subset S_d$, where $B_0=\emptyset$. Just as before, $\delta \left(A_k / S_d \right)>0$ and $\delta \left(S_d /B_i \right)<0$, meaning that $D^n_j>0$.

Again, in the rigid case the same reasoning gives that $D^n_n>0$, meaning that after summing over permutations $v_1, \ldots, v_n$ we have $C_H>0$.\\

This completes our analysis of integrals (\ref{equation:multint}) and (\ref{equation:multint3}), establishing that both satisfy condition (\ref{equation:Theta}).


\begin{thebibliography}{99} 

\bibitem{BA} Albert, R., \& Barabási, A. L. (2002). Statistical mechanics of complex networks. Reviews of modern physics, 74(1), 47.

\bibitem{BS} Baldwin, J., \& Shelah, S. (1997). Randomness and semigenericity. Transactions of the American Mathematical Society, 349(4), 1359-1376.

\bibitem{BShi} Baldwin, J. T., \& Shi, N. (1996). Stable generic structures. Annals of Pure and Applied Logic, 79(1), 1-35.

\bibitem{BT} Bollob\'as, B., \& Thomason, A. (1997). Hereditary and monotone properties of graphs. In The Mathematics of Paul Erd\H{o}s II (pp. 70-78). Springer, Berlin, Heidelberg.

\bibitem{Cam} Cameron, P. J. (1997). The random graph. In The Mathematics of Paul Erdös II (pp. 333-351). Springer, Berlin, Heidelberg.

\bibitem{EGIS} Elkind, E., Gan, J., Igarashi, A., Suksompong, W., \& Voudouris, A. A. (2019). Schelling Games on Graphs. arXiv preprint arXiv:1902.07937.

\bibitem{E1} Elwes, R. (2015). Preferential Attachment Processes Approaching The Rado Multigraph. arXiv preprint arXiv:1502.05618

\bibitem{E2} Elwes, R. (2015). A Linear Preferential Attachment Process Approaching the Rado Graph. arXiv preprint arXiv:1502.05618 

\bibitem{ER} Erd\H{o}s, P., \& R\'enyi, A. (1960). On the evolution of random graphs. Publ. Math. Inst. Hung. Acad. Sci, 5(1), 17-60.

\bibitem{ER2} Erd\H{o}s, P., \& R\'enyi, A. (1963). Asymmetric graphs. Acta Mathematica Academiae Scientiarum Hungarica, 14(3-4), 295-315.

\bibitem{FK} Frieze, A., \& Karo\'nski, M. (2015). Introduction to random graphs. Cambridge University Press.


\bibitem{GWOL} Gleeson, J. P., Ward, J. A., O'Sullivan, K. P., \& Lee, W. T. (2013). Competition-induced criticality in a model of meme popularity. arXiv preprint arXiv:1305.4328.

\bibitem{EH} Hrushovski, E. (1993). A new strongly minimal set, Annals of Pure and Applied Logic 62, 147-166



\bibitem{KV} Kim, J. H., \& Vu, V. H. (2000). Concentration of multivariate polynomials and its applications. Combinatorica, 20(3), 417-434.

\bibitem{KRRSTU} Kumar, R., Raghavan, P., Rajagopalan, S., Sivakumar, D., Tomkins, A., \& Upfal, E. (2000). Stochastic models for the web graph. In Proceedings 41st Annual Symposium on Foundations of Computer Science (pp. 57-65). IEEE.


\bibitem{KK} Kleinberg, R. \& Kleinberg, J. \newblock Isomorphism and embedding problems for infinite limits of scale-free graphs.  \newblock In Proceedings of the sixteenth annual ACM-SIAM symposium on Discrete algorithms (SODA) (pp. 277-286). Society for Industrial and Applied Mathematics. (2005)

\bibitem{EHN} Lieberman, E., Hauert, C., \& Nowak, M. A. (2005). Evolutionary dynamics on graphs. Nature, 433(7023), 312.

\bibitem{PSV} Pastor-Satorras, R., \& Vespignani, A. (2001). Epidemic spreading in scale-free networks. Physical review letters, 86(14), 3200.

\bibitem{PPP} Proulx, S. R., Promislow, D. E., \& Phillips, P. C. (2005). Network thinking in ecology and evolution. Trends in ecology \& evolution, 20(6), 345-353.


\bibitem{JS} Spencer, J. (2001). The strange logic of random graphs (Vol. 22). Springer Science \& Business Media

\bibitem{SS} Shelah, S., \& Spencer, J. (1988). Zero-one laws for sparse random graphs. Journal of the American Mathematical Society, 1(1), 97-115.

\bibitem{VR} Sood, V., \& Redner, S. (2005). Voter model on heterogeneous graphs. Physical review letters, 94(17), 178701.

\bibitem{WS} Watts, D. J., \& Strogatz, S. H. (1998). Collective dynamics of ‘small-world’ networks. nature, 393(6684), 440.

\end{thebibliography}
\end{document}